\documentclass{amsart}
\usepackage{amssymb}

\usepackage{amsmath}
\usepackage{amscd}

\newtheorem{theorem}{Theorem}
\newtheorem{lemma}{Lemma}

\newtheorem{corollary}{Corollary}

\theoremstyle{definition}
\newtheorem{definition}{Definition}
\newtheorem{remark}{Remark}
\newtheorem{example}{Example}
\newtheorem{question}{Question}

\newcommand{\p}{{\mathbb P}}

\newcommand{\g}{{\mathbb G}}

\renewcommand{\L}{{\mathcal L}}
\renewcommand{\O}{{\mathcal O}}

\newcommand{\cd}{\operatorname{codim}}

\begin{document}
\title{A bound on the degree of schemes defined by quadratic equations}
\author[Alberto Alzati \and Jos\'{e} Carlos Sierra]{Alberto Alzati* \and Jos\'{e} Carlos Sierra**}
\address{Dipartimento di Matematica, Universit\`a degli Studi di Milano\\
via Cesare Saldini 50, 20133 Milano (Italy)}
\email{alberto.alzati@unimi.it}

\address{Instituto de Ciencias Matem\'aticas (ICMAT), Campus de
Cantoblanco, Carretera de Colmenar km.15, 28049 Madrid (Spain)}
\email{jcsierra@icmat.es}
\thanks{* This work is within the framework of the national research project
``Geometry on Algebraic Varieties'' Cofin 2006 of MIUR}

\thanks{**
Research partially supported by the Spanish projects MTM2006-04785
and MTM2009-06964, as well as by the mobility programs ``Profesores
de la UCM en el extranjero. Convocatoria 2008" and ``Jos\'e
Castillejo", MICINN grant JC2009-00098}
\date{\today}
\begin{abstract}
We consider complex projective schemes $X\subset\Bbb{P}^{r}$ defined
by quadratic equations and satisfying a technical hypothesis on the
fibres of the rational map associated to the linear system of
quadrics defining $X$. Our assumption is related to the syzygies of
the defining equations and, in parti\-cular, it is weaker than
properties $N_2$, $N_{2,2}$ and $K_2$. In this setting, we show that
the degree, $d$, of $X\subset\Bbb{P}^{r}$ is bounded by a function
of its codimension, $c$, whose asymptotic behaviour is given by
${2^c}/{\sqrt[4]{\pi c}}$, thus improving the obvious bound $d\leq
2^c$. More precisely, we get the bound
$\binom{d}{2}\leq\binom{2c-1}{c-1}$. Furthermore, if $X$ satis\-fies
property $N_p$ or $N_{2,p}$ we obtain the better bound
$\binom{d+2-p}{2}\leq\binom{2c+3-2p}{c+1-p}$. Some classification
results are also given when equality holds.


\end{abstract}
\maketitle

\section{Introduction}

The equations defining a projective variety (or scheme) and the
syzygies among them play a central role in algebraic geometry. From
this point of view, perhaps the most interesting case is that of
quadratic equations.

In modern terms, the study of varieties defined by quadratic
equations was initiated in Mumford's foundational paper \cite{m},
where it is proved that a multiple $|m\L|$ of any ample line bundle
$\L$ over an algebraic variety $X$ gives an embedding in a
projective space which is defined by quadrics if $m$ is big enough.
Moreover, effective values of $m$ were also obtained for curves and
abelian varieties. Let us look more closely at the case of curves. A
classical theorem of Castelnuovo \cite{c} (also attributed to
Mattuck \cite{mattuck} and Mumford \cite{m}) states that a curve $X$
of genus $g$ embedded in projective space by a complete linear
system $|\L|$ is projectively normal if $\deg(\L)\geq 2g+1$, and a
theorem due to Fujita \cite{f} and Saint-Donat \cite{sd},
strengthening earlier work of Mumford, asserts that the homogeneous
ideal of $X$ is generated by quadrics if $\deg(\L)\geq 2g+2$. These
results were generalized to syzygies by Green, proving that $X$
satisfies property $N_p$ if $\deg(\L)\geq 2g+1+p$ (see \cite{g}).
Since then, many efforts have been made to extend this type of
result on property $N_p$ to other varieties.

Bearing in mind that \emph{any} algebraic variety $X$ admits an
embedding into $\Bbb{P}^{r}$ such that its homogeneous ideal is
generated by quadrics, a natural question arises:

\begin{center}
\emph{What can be said about the degree of $X\subset\Bbb{P}^{r}$?}
\end{center}

The aim of this paper is to obtain a bound on the degree, $d$, of a
scheme $X\subset\Bbb{P}^{r}$ defined by quadrics in terms of its
codimension $c$. It is obvious that $d\leq 2^c$, with equality if
$X\subset\Bbb{P}^{r}$ is the complete intersection of $c$
independent quadric hypersurfaces. In this case, the number of the
defining equations is minimal with respect to $c$. On the other
side, Zak showed in \cite{z} that if $X\subset\Bbb{P}^{r}$ is either
a non-degenerate integral subvariety, or a finite set of points in
general position, and the number of independent quadratic equations
of $X$ is almost maximal with respect to $c$, then $d\leq 2c$ (cf.
Remark \ref{rem:zak}). But, to our best knowledge, very little is
known about the degree of $X\subset\Bbb{P}^{r}$ if the number of the
independent quadratic defining equations is neither minimal nor
maximal with respect to $c$ (cf. Remark \ref{rem:egh}).

However, our approach to this matter is a little bit different from
that of \cite{z}, even if we also use elementary techniques of
projective geometry that allow us to work in a very general setting:
let $\Lambda\subset H^{0}(\Bbb{P}^{r},\mathcal{O}_{\Bbb{P}^{r}}(2))$
be a linear subspace of dimension $\alpha +1$ and let
$X\subset\Bbb{P}^{r}$ denote the base scheme of $\Lambda$. Let
$\Phi:\Bbb{P}^{r}\backslash X\to \Bbb{P}^{\alpha}$ be the morphism
given by $\Lambda$. This map has been widely studied from many
points of view (see, for instance, \cite{v} and references therein).
Our results are obtained under an assumption involving the map
$\Phi$. We impose an upper bound on the dimension of the set
$W\subset G(1,r)$ consisting of lines $L\subset\Bbb{P}^{r}$ for
which the restriction $\Phi_{|L}:L\to\Phi(L)$ is a double covering
(see Lemma \ref{lemL} and Definition \ref{defW}). More precisely, we
assume $\dim(W)\leq 2n+1$, where $n:=\dim(X)$.

The main results of the paper are summarized in the following
theorem. First, we introduce some terminology. We say that $X$ is
\emph{reduced in codimension zero} if the scheme-theoretic
intersection of $X$ with a general linear subspace of $\Bbb{P}^{r}$
of codimension $n$ is reduced, and we say that $X$ is \emph{smooth
and integral in codimension one} if $n\geq 1$ and the
scheme-theoretic intersection of $X$ with a general linear subspace
of $\Bbb{P}^{r}$ of codimension $n-1$ is a smooth integral curve.

\begin{theorem}\label{thm:main}
Let $X\subset\Bbb{P}^{r}$ be a (possibly singular, non-reduced,
reducible or non-equidimensional) complex projective scheme of
degree $d$ and dimension $n$ defined by a linear system of quadrics
$\Lambda\subset H^{0}(\Bbb{P}^{r},\mathcal{O}_{\Bbb{P}^{r}}(2))$.
Assume that $X$ is reduced in codimension zero. Let $\alpha:=\dim
(\Lambda)-1$ and let $c:=r-n\geq 2$. Let $\Phi:\Bbb{P}^{r}\backslash
X\to \Bbb{P}^{\alpha}$ be the morphism given by $\Lambda $ and let
$W\subset G(1,r)$ be the closure of the set of lines
$L\subset\Bbb{P}^{r}$ for which the restriction
$\Phi_{|L}:L\to\Phi(L)$ is a double covering. The following holds:
\begin{itemize}
\item[(i)] if $\dim (W)\leq 2n+1$ then $\alpha \geq 2c-2$ and $\binom{d}{2}\leq \binom{2c-1}{c-1}$. Furthermore, if $\dim (W)\leq 2n$ then $\binom{d}{2}=\binom{2c-1}{c-1}$ if and only if $\alpha=2c-2$.
\item[(ii)] if $\dim (W)\leq 2n$, $\binom{d}{2}=\binom{2c-1}{c-1}$ and, moreover, $X$ is smooth and integral in codimension one then
either $d=3$, $c=2$ and $g=0$ or $d=5$, $c=3$ and $g=1$, where $g$
denotes the sectional genus of $X\subset\p^r$.
\end{itemize}
\end{theorem}

We remark that our hypothesis is quite general. In fact, the
assumption on $W$ is closely related to the syzygies of the defining
equations of $X$. For instance, if the trivial (or Koszul) relations
among the elements of $\Lambda$ are generated by linear syzygies,
then the closure of any fibre of $\Phi$ is a linear space and
$W=\emptyset$. This condition is called $K_2$ and it was introduced
by Vermeire in \cite{v}. Condition $K_2$ is weaker than the deeply
studied property $N_p$ defined in \cite{g} (see \cite{v}). In fact,
for any $p\geq 2$ we have
$$N_p\Rightarrow N_{2,p}\Rightarrow
K_2\Rightarrow W=\emptyset$$ (see \cite{eghp} for the definition and
results on property $N_{2,p}$). So, in particular, Theorem
\ref{thm:main} yields the following:

\begin{corollary}\label{cor:main}
Let $X\subset\Bbb{P}^{r}$ be a (possibly singular, non-reduced,
reducible or non-equidimensional) complex projective scheme of
degree $d$ and dimension $n$ defined by a linear system of quadrics
$\Lambda\subset H^{0}(\Bbb{P}^{r},\mathcal{O}_{\Bbb{P}^{r}}(2))$.
Assume that $X$ is reduced in codimension zero. Let $\alpha:=\dim
(\Lambda)-1$ and let $c:=r-n\geq 2$. If $X$ satisfies condition
$K_{2}$ then:
\begin{itemize}
\item[(i)] $\alpha \geq 2c-2$ and $\binom{d}{2}\leq \binom{2c-1}{c-1}$. Furthermore, $\binom{d}{2}=\binom{2c-1}{c-1}$ if and only if $\alpha=2c-2$.
\item[(ii)] if moreover $X$ is smooth and integral in codimension one and $\binom{d}{2}=\binom{2c-1}{c-1}$, then
either $d=3$, $c=2$ and $g=0$ or $d=5$, $c=3$ and $g=1$, where $g$
denotes the sectional genus of $X\subset\p^r$.
\end{itemize}
\end{corollary}

Thus, as an application of these results, we obtain some numerical
conditions for a given variety (or scheme) to satisfy properties
$N_2$, $N_{2,2}$ or $K_2$. For instance, as an immediate consequence
of Corollary \ref{cor:main}, we show in Example \ref{ex:8points}
that $8$ general points in $\Bbb{P}^{4}$ cannot satisfy property
$N_2$. This is an example where \cite[Theorem 1]{gl} is sharp for
$p=2$. Furthermore, recent work of Han and Kwak shows that if
$X\subset\p^r$ is a reduced scheme then property $N_{2,p}$ behaves
well under inner projections (see \cite{h-k}). Thanks to their
result, we can improve Corollary \ref{cor:main} in the following
way:

\begin{corollary}\label{cor:Np}
Let $X\subset\Bbb{P}^{r}$ be a (possibly singular, reducible or
non-equidimensional) reduced complex projective scheme of degree
$d$, dimension $n$ and codimension $c$ satisfying property $N_{p}$
or $N_{2,p}$ for some $p\geq 2$. Then:
\begin{itemize}
\item[(i)] $h^0(\Bbb{P}^{r},\mathcal{I}_X(2))\geq cp-\binom{p}{2}$ and
$\binom{d+2-p}{2}\leq \binom{2c+3-2p}{c+1-p}$. Furthermore,
$\binom{d+2-p}{2}=\binom{2c+3-2p}{c+1-p}$ if and only if
$h^0(\Bbb{P}^{r},\mathcal{I}_X(2))=cp-\binom{p}{2}$.
\item[(ii)] if moreover $X$ is smooth and integral in codimension one and $\binom{d+2-p}{2}=\binom{2c+3-2p}{c+1-p}$, then
either $d=p+1$, $c=p$ and $g=0$ or $d=p+3$, $c=p+1$ and $g=1$, where
$g$ denotes the sectional genus of $X\subset\p^r$.
\end{itemize}
\end{corollary}

The main idea of the proof of Theorem \ref{thm:main} is the
following. We choose a suitable smooth subvariety
$Y\subset\Bbb{P}^{r}$ of dimension $c-1$ and we estimate, \emph{in
two different ways}, the number of double points of the morphism
$\Phi_{|Y}:Y\to\Phi(Y)$. This clearly explains why we need to assume
$\dim(W)\leq 2n+1$. Otherwise, we cannot guarantee that the number
of double points of the morphism $\Phi_{|Y}:Y\to\Phi(Y)$ is finite.

The paper is structured as follows. In Section
\ref{section:notation} we fix notation. In Section
\ref{section:preliminary} we obtain two results, on the number of
common secant lines to a pair of smooth subvarieties
$X,Y\subset\Bbb{P}^{r}$ and on the number of apparent double points
of the $2$-Veronese embedding of a smooth subvariety
$Y\subset\Bbb{P}^{r}$, that we use in the sequel to compute the
number of double points of $\Phi_{|Y}$. Finally, in Section
\ref{section:main} we include the proofs of Theorem \ref{thm:main}
and Corollaries \ref{cor:main} and \ref{cor:Np}, as well as some
related examples and remarks. Theorem \ref{thm:main} and Corollary
\ref{cor:main} follow from Theorems \ref{teofond} and
\ref{teofondter}, and Corollaries \ref{corK2} and \ref{corfin},
respectively.

\section{Notation}\label{section:notation}

We work over the field of complex numbers. We will adopt the
following notation:

$\Bbb{P}^{r}$: $r$-dimensional projective space

$\Lambda$: linear subspace of
$H^{0}(\Bbb{P}^{r},\mathcal{O}_{\Bbb{P}^{r}}(2))$ generated by the
elements $f_{0},\dots,f_{\alpha}$

$\alpha$: $\dim (\Lambda )-1$


$X$: subscheme of $\Bbb{P}^{r}$ defined by $f_{0},\dots,f_{\alpha}$,
in brief: defined by $\Lambda$

$d$: degree of $X$ in $\Bbb{P}^{r}$

$n$: dimension of $X$

$c$: codimension of $X$ in $\Bbb{P}^{r}$

$g$: sectional genus of $X\subset\p^r$, if $X$ is smooth and
integral in codimension one

$\Phi$: rational map from $\Bbb{P}^{r}$ to
$\Bbb{P}^{\alpha}:=\Bbb{P}(\Lambda ^{*})$, induced by the linear
system $\Lambda$ in $\Bbb{P}^{r}$





$N(r):=h^{0}(\Bbb{P}^{r},\mathcal{O}_{\Bbb{P}^{r}}(2))-1$

$v_{2}$: Veronese map from $\Bbb{P}^{r}$ to $\Bbb{P}^{N(r)}$ given
by the complete linear system $|\mathcal{O}_{\Bbb{P}^{r}}(2)|$

$\langle Z\rangle $: linear span of a subvariety $Z$ embedded in a
projective space.


\section{Preliminary results}\label{section:preliminary}


Let us compute the number of common secant lines to smooth integral
subvarieties $X,Y\subset\Bbb{P}^{r}$ in terms of the number of
apparent double points of their respective linear sections. We prove
a suitable generalization of the formula \cite[p. 297]{gh} for two curves in
$\Bbb{P}^{3}$.

We say that $X,Y\subset\Bbb{P}^r$ are \emph{in general position} if
the set of common secant lines to $X$ and $Y$ has the expected
dimension.

\begin{lemma}\label{lemsec}
Let $X,Y$ be two smooth integral non-linear subvarieties in
$\Bbb{P}^{r}$ of dimension $h,k$ and degree $d,\delta$,
respectively, such that $h+k=r-1$. Assume that $X$ and $Y$ are in
general position in $\Bbb{P}^{r}$. For $i=0,\dots,h$, let $a_{i}$ be
the number of double points of a generic projection in
$\Bbb{P}^{2i}$ of a generic linear section of $X$ of dimension $i$
(for $i=h$ we consider $X$, and for $i=0$ we have
$a_{0}=\binom{d}{2}$). For $ i=0,\dots,k$, let $b_{i}$ be the number
of double points of a generic projection in $\Bbb{P}^{2i}$ of a
generic linear section of $Y$ of dimension $i$ (for $i=k$ we
consider $Y$, and for $i=0$ we have $b_{0}=\binom{\delta}{2}$). Then
the number of lines in $\Bbb{P}^{r}$ which are secant to both $X$
and $Y$ is $\sum\limits_{i=0}^{\min \{h,k\}}a_{i}b_{i}$.
\end{lemma}

\begin{proof}
Since $X,Y\subset\Bbb{P}^r$ are in general position and
$h+k=r-1$, we get a finite number $\varkappa $ of lines in
$\Bbb{P}^{r}$ which are secant to both $X$ and $Y$. We can assume
that $h\leq k.$ In this case we have to show that $\varkappa
=\sum\limits_{i=0}^{h}a_{i}b_{i}.$

Let $G$ be the Grassmannian $G(1,r)$ of lines of $\Bbb{P}^{r}$, $\dim
(G)=2r-2.$ The secant lines of $X$ give rise to a subvariety $S_{X}\subset G$
of dimension $2h.$ The secant lines of $Y$ give rise to a subvariety $%
S_{Y}\subset G$ of dimension $2k.$ Obviously $\varkappa =S_{X}S_{Y}$ in the
cohomology ring $H^{*}(G,\Bbb{Z})$ of $G.$ It is well known that $H^{*}(G,%
\Bbb{Z})$ is generated by the cohomology classes of the so-called
Schubert cycles $\Omega (p,q),$ where $\Omega (p,q)$ is the
subvariety of $G$ parametrizing the lines of $\Bbb{P}^{r}$ contained
in a generic linear
subspace $B$ of dimension $q$ and intersecting a generic linear subspace $%
A\subset B$ with $\dim (A)=p.$ As $\dim [\Omega (p,q)]=p+q-1$ we have
\begin{center}
$S_{X}=\sum\limits_{i=0}^{h}\alpha _{i}\Omega (i,2h+1-i)$

$S_{Y}=\sum\limits_{j=0}^{h}\beta _{j}\Omega (2k+1-r+j,r-j)$
\end{center}
for suitable $\alpha _{i}, \beta _{j}\in \Bbb{Z}.$ By recalling that
$h+k=r-1$ we can write
\begin{center}
$S_{Y}=\sum\limits_{j=0}^{h}\beta _{j}\Omega (r-(2h+1)+j,r-j).$
\end{center}
Now let us remark that $\Omega (i,2h+1-i)\Omega (r-(2h+1)+j,r-j)=\delta
_{j}^{i}$ (Kronecker symbols) so that $S_{X}S_{Y}=$ $\sum\limits_{i=0}^{h}%
\alpha _{i}\beta _{i}.$ Moreover:

$\alpha _{h}=S_{X}\Omega (r-(h+1),r-h)$ is the number of secant
lines to $X$ contained in a generic linear subspace
$B\subset\Bbb{P}^{r}$ of dimension $r-h$. As $\dim (X)=h$, this number is $\binom{d}{2}=a_{0}$.

$\alpha _{h-1}=$ $S_{X}\Omega (r-(h+2),r-h+1)$ is the number of
secant lines to $X$ contained in a generic linear subspace
$B\subset\Bbb{P}^{r}$ of dimension $r-h+1$ (cutting $X$ along a
generic section of dimension $1$) and intersecting a generic linear
subspace $A\subset B$ of dimension $r-h-2,$ i.e. it is the number of
double points of a generic projection in $\Bbb{P}^{2}$ of the curve
$X\cap B,$ hence $\alpha _{h-1}=a_{1}$. And so on until $\alpha_{0}=a_{h}$.

For $Y$ we can argue in the same way.
\end{proof}







We now obtain the number of apparent double points of the
$2$-Veronese embedding of a smooth subvariety $Y\subset\Bbb{P}^{r}$ in terms of
the number of apparent double points of the linear sections of $Y$.

\begin{theorem}\label{formula}
Let $Y$ be a smooth integral subvariety of $\Bbb{P} ^{r}$ of
dimension $k$ and degree $\delta$. Let $v_{2}(Y)$ be the
$2$-Veronese embedding of $Y$ in $\Bbb{P}^{N(r)}$, where
$N(r):=h^{0}(\Bbb{P}^{r},\mathcal{O}_{\Bbb{P}^{r}}(2))-1$. Let
$\Delta [v_{2}(Y)]$ be the number of double points of a generic
projection of $v_{2}(Y)$ in $\Bbb{P}^{2k}$. For $i=0,\dots,k$, let
$b_{i}$ be the number of double points of a generic projection in
$\Bbb{P}^{2i}$ of a generic linear section of $Y$ of dimension $i$
(for $i=k$ we consider $Y$, and for $i=0$ we have
$b_{0}=\binom{\delta}{2}$). Then

\begin{center}
$\Delta [v_{2}(Y)]=\sum\limits_{i=0}^{k}\binom{2k+1}{k-i}b_{i}.$
\end{center}
\end{theorem}

\begin{proof}
If $k=0$ we have nothing to prove. If $k\geq 1,$ let $Y=Y_{k}\supset
Y_{k-1}\supset \dots\supset Y_{1}\supset Y_{0}$ be a ladder of
smooth hyperplane sections with $\dim (Y_{k-q})=k-q$,
$q=0,1,\dots,k$. Let $H$ be the class of the hyperplane divisor of
$Y$. Let $\delta :=H^{k}$ be the degree of $Y$. First of all we need
a formula for the Segre classes of any $Y_{k-q}$. Let $s_{i}$ be the
$i$-th Segre class of $Y,$ then we have

\begin{center}
$s_{i}(Y_{k-q})=\sum\limits_{j=0}^{q}\binom{q}{j}H^{q+j}s_{i-j}$ (with $%
s_{i-j}=0$ if $i<j$) $\qquad (*).$
\end{center}

By using \cite[Theorem 3.4]{ps} we have
$$
\begin{array}{lcl}

b_{k} & = & \frac{1}{2}\left( \delta ^{2}-\sum\limits_{i=0}^{k}\binom{2k+1}{k-i}%
H^{k-i}s_{i}\right) ;\\

b_{k-1} & = & \frac{1}{2}\left( \delta ^{2}-\sum\limits_{i=0}^{k-1}\binom{2(k-1)+1%
}{k-1-i}H^{k-1-i}s_{i}(Y_{k-1})\right) ;\\


\vdots & & \\

b_{1} & = & \frac{1}{2}\left( \delta ^{2}-\sum\limits_{i=0}^{1}\binom{3}{1-i}%
H^{1-i}s_{i}(Y_{1})\right) ;\\

b_{0} & = & \frac{1}{2}\left( \delta ^{2}-H^{k}s_{0}\right)\, =\,
\binom{\delta}{2} .
\end{array}
$$

On the other hand, the number of the double points of a generic
projection of $v_{2}(Y)$ in $\Bbb{P}^{2k}$ is

\begin{center}
$\Delta [v_{2}(Y)]$ $=\frac{1}{2}\left( 2^{2k}\delta
^{2}-\sum\limits_{i=0}^{k}\binom{2k+1}{k-i}2^{k-i}H^{k-i}s_{i}\right) .\;$
\end{center}

Let us suppose that this number is equal to
$a_{k}b_{k}+a_{k-1}b_{k-1}+\dots+a_{1}b_{1}+a_{0}b_{0}$ for a
suitable choice of $a_{i}\;$and let us try to find $a_{i.}$ For
instance, by considering the coefficients of $\delta ^{2}$ we get

\begin{center}
$a_{k}+a_{k-1}+\dots+a_{1}+a_{0}$ $=2^{2k}$ \qquad $(**).$
\end{center}

Moreover we must have

\begin{center}
$\sum\limits_{i=0}^{k}\binom{2k+1}{k-i}2^{k-i}H^{k-i}s_{i}=\sum%
\limits_{i=0}^{k}a_{k-i}b_{k-i}^{\prime }$ \qquad $(***),$
\end{center}
where $b_{k-i}^{\prime }:=\delta ^{2}-2b_{k-i}.$

If we consider the coefficients of $H^{k-i}s_{i},$ $i=0,1,\dots,k,$
in $(***),$ taking care of the relations $(*),$ we get a system of
$k+1$ linear equations in the $k+1$ unknowns $a_{i},$
$i=0,1,\dots,k,$ whose associated matrix is triangular, with a
non-zero determinant.

We claim that the only solution of this system of linear equations
is $a_{i}=\binom{2k+1}{k-i}$, $i=0,1,\dots,k$. To see this fact we
transform $(***)$ into

\begin{center}
$\sum\limits_{i=0}^{k}\binom{2k+1}{k-i}(2^{k-i}-1)H^{k-i}s_{i}=\sum%
\limits_{i=1}^{k}a_{k-i}b_{k-i}^{\prime }$
\end{center}
and we write (by using the relations $(*)$)

\begin{center}
$\sum\limits_{i=0}^{k}\binom{2k+1}{k-i}(2^{k-i}-1)H^{k-i}s_{i}=\sum%
\limits_{q=1}^{k}\sum\limits_{i=0}^{k}\binom{2k+1}{k-i}\binom{k-i}{q}%
H^{k-i}s_{i}$ $\quad (i)$

$\sum\limits_{i=1}^{k}\binom{2k+1}{i}b_{k-i}^{\prime
}=\sum\limits_{q=1}^{k}\sum\limits_{p=0}^{k-q}\sum\limits_{j=0}^{q}\binom{%
2k+1}{q}\binom{2(k-q)+1}{k-q-p}\binom{q}{j}H^{k-(p-j)}s_{p-j}$ $\quad (ii)$
\end{center}
where $s_{p-j}=0$ when $p<j.$ Now we fix any $q=1,2,\dots,k$ and we
look at the coefficients of $H^{k-i}s_{i},$ $i=0,1,\dots,k.$

In $(i)$ the coefficient of $H^{k-i}s_{i}$ is
$\binom{2k+1}{k-i}\binom{k-i}{q}$. In $(ii)$ it is

\begin{center}
$\binom{2k+1}{q}\sum\limits_{j=0}^{\min (q,k-q-i)}\binom{2(k-q)+1}{k-q-i-j}%
\binom{q}{j}=\binom{2k+1}{q}\binom{2(k-q)+1+q}{k-q-i}$,
\end{center}
where the last equality is the well-known formula
$\sum\limits_{j=0}^{\min
(a,b)}\binom{m}{b-j}\binom{a}{j}=\binom{m+a}{b}$.

Now it is easy to check that

\begin{center}
$\binom{2k+1}{k-i}\binom{k-i}{q}=\binom{2k+1}{q}\binom{2k-q+1}{k-q-i}.$
\end{center}

Therefore $a_{i}=\binom{2k+1}{k-i},$ with $i=0,1,\dots,k$, is the
only solution of $(***)$ and obviously it is a solution for $(**)$
too.
\end{proof}

In particular, the following computation will be very useful for our purposes.

\begin{corollary}\label{corQ}
Let $Y$ be a smooth $k$-dimensional quadric of $\Bbb{P}^{r}$. Let
$v_{2}(Y)$ be the $2$-Veronese embedding of $Y$ in $\Bbb{P}^{N(r)}$.
Then $\Delta [v_{2}(Y)]=\binom{2k+1}{k}$.
\end{corollary}

\begin{proof}
Let us apply Theorem \ref{formula}: in this case $b_{0}=1$ and $b_{i}=0$ for
$i=1,\dots,k$ because the generic $i$-dimensional linear section of $Y$ has no apparent double points.
\end{proof}

\section{Main results}\label{section:main}

Let $\Phi:\Bbb P^r\dashrightarrow\Bbb P^{\alpha}$ be the rational map given by a linear system of quadrics $\Lambda$.
Let us analyze the restriction of $\Phi$ to a line $L\subset\Bbb{P}^{r}$.

\begin{lemma}
\label{lemL} Let $\Lambda $ be a linear system of quadrics in
$\Bbb{P}^{r}$, let $\Phi $ be the associated rational map and let
$X$ be the scheme defined by $\Lambda$. Consider the restriction
$\Phi _{|L}$ of the rational map $\Phi$ to a line $L$ of $\Bbb P^r$.
Then one of the following holds:

$(i)$ $L$ is contained in $X$ and $\Phi_{|L}$ is not defined;

$(ii)$ $L$ is a secant line for $X$ and $\Phi_{|L}$ contracts $L$ to
a point;

$(iii)$ $L$ intersects $X$ at one point and $\Phi_{|L}$ can be
extended to an isomorphism among $L$ and $\Phi(L)$;

$(iv)$ $L\cap X=\emptyset $ and $\Phi_{|L}$ is a Veronese embedding
of $L$;

$(v)$ $L\cap X=\emptyset $ and $\Phi_{|L}$ is a double covering of
$\Phi(L)$ by $L$.
\end{lemma}

\begin{proof}
Let us choose coordinates $(x:y)$ on $L$ and let us consider that
$\Phi _{|L}$ is given by a linear system of quadrics on
$\Bbb{P}^{1}$ of the following type: $\lambda
_{i}x^{2}+\mu_{i}xy+\nu_{i} y^{2}$, $i=0,\dots,\alpha$. In case
$(i)$ all quadrics vanish identically. In case $(ii)$, let $(1:0)$
and $(0:1)$ be the coordinates of the two points $X\cap L$. All
quadrics are of the following type: $\mu _{i}xy,$ $
i=0,\dots,\alpha$, so that $\Phi(L)=(\mu _{0}:\dots:\mu_{\alpha})$.
In case $(iii)$, let $(1:0)$ be the coordinates of $X\cap L$. All
quadrics are of the following type: $\mu _{i}xy+\nu_{i}
y^{2}=(\mu_{i}x+\nu_{i} y)y$, $i=0,\dots,\alpha$, so that
$\Phi_{|L}$ can be extended to the isomorphism given by:
$\mu_{i}x+\nu_{i} y$, $i=0,\dots,\alpha$. In cases $(iv)$ and $(v)$,
$ \Phi_{|L}$ is a morphism given by a base point free linear system
of degree two on $L$. If the linear system is complete, then
$\Phi_{|L}$ is a Veronese embedding of $\Bbb P^1$. Otherwise, it is a double
covering of $\Bbb P^1$.
\end{proof}

\begin{lemma}\label{lemW}
In the setting of Lemma \ref{lemL}, a line $L$ such
that $L\cap X=\emptyset $ yields case $(v)$ if and only if it is a
secant line to a non-linear fibre of $\Phi$.
\end{lemma}

\begin{proof}
Let $\Phi_{P}:=\overline{\Phi^{-1}(\Phi(P))}$ be a non-linear fibre
of $\Phi$ for some point $P\in \Bbb{P}^{r}\backslash X.$ Let $Q$,
$Q'$ be any two points of $\Phi_{P}$ and let $L$ be the line
$\langle Q,Q'\rangle $. Let us assume that $L\cap X=\emptyset$ and
let us consider $\Phi_{|L}$. As $\Phi(Q)=\Phi(Q')$, $\Phi_{|L}$
cannot be an embedding so that case $(v)$ holds. Note that this is
also true when $Q=Q^{\prime}$, i.e. for tangent lines to $\Phi_{P}$.

On the other hand, let $L$ be a line in $\Bbb{P}^{r}$ such that
$L\cap X=\emptyset $ and for which case $(v)$ holds. Let $Q,Q'\in L$
be two points such that $\Phi(Q)=\Phi(Q')$. These points belong to
some fibre $\Phi_{P}$ that cannot be a linear space, otherwise $L$
would be contained in $\Phi_{P}$ and case $(ii)$ would hold. Hence
$L$ is a secant line for $\Phi_{P}$.
\end{proof}

\begin{definition}\label{defW}
In the setting of Lemma \ref{lemL}, let us consider the set
$\mathcal{W}\subset G(1,r)$ consisting of lines $L$ in $\Bbb{P}^{r}$
for which case $(v)$ holds. Let $W$ be the Zariski closure of
$\mathcal{W}$.
\end{definition}

\begin{corollary}
\label{corlin} $W=\emptyset$ if and only if all fibres of $\Phi $
are linear spaces.
\end{corollary}

\begin{proof}
This is just a reformulation of Lemma \ref{lemW}.
%
%
\end{proof}




Now we can prove the main results of the paper:

\begin{theorem}
\label{teofond} Let $X\subset\Bbb{P}^{r}$ be a scheme of degree $d$
and dimension $n$ defined by a linear system of quadrics $\Lambda$.
Assume that $X$ is reduced in codimension zero. Let
$\alpha:=\dim(\Lambda)-1$ and let $c:=r-n\geq 2$. If $\dim(W)\leq
2n+1$, then $\alpha \geq 2c-2$ and $\binom{d}{2}\leq
\binom{2c-1}{c-1}$. Furthermore, if $\dim(W)\leq
2n$ then $\binom{d}{2}=\binom{2c-1}{c-1}$ if
and only if $\alpha=2c-2$.
\end{theorem}

\begin{proof}
Let $\Phi:\Bbb{P}^r\dashrightarrow\Bbb{P}^{\alpha}$ be the rational
map given by $\Lambda$, and let $v_{2}:\Bbb{P}^r\to\Bbb{P}^{N(r)}$
be the $2$-Veronese embedding of $\Bbb{P}^r$. Then $\Phi =\pi \circ
v_{2},$ as rational maps, where
$\pi:\Bbb{P}^{N(r)}\dashrightarrow\Bbb P^{\alpha}$ is the projection
from the linear subspace $B:=\langle v_2(X)\rangle $ of codimension
$\alpha+1$ in $\Bbb{P}^{N(r)}$ and $N(r):=\binom{r+2}{2}-1$.
Furthermore, we remark that $B\cap v_{2}(\Bbb{P}^{r})=v_{2}(X)$ as
$X$ is defined by $\Lambda$.

Let us choose a generic $c$-dimensional linear space $A\subset
\Bbb{P}^{r}$ such that $A\cap X$ is given by $d$ distinct points
$x_{1},\dots,x_{d}$. This is possible since we assume $X$ is reduced
in codimension zero. Let $Y\subset A$ be a smooth quadric of
dimension $c-1$ such that $Y\cap X=\emptyset$. Then $\Phi_{|Y}$ is a
regular map. We claim that $\Phi _{|Y}:Y\to\Phi(Y)$ is a finite
morphism. If $\Phi_{|Y}$ has a positive dimensional fibre over a
point $R\in \Bbb{P}^{\alpha }$, then we get also a positive
dimensional fibre of $\pi _{|v_{2}(Y)}$ over the same point $R$, but
this fibre is contained in $\langle B,R\rangle $, then it intersects
$B$. Hence we would have $v_{2}(Y)\cap B\neq \emptyset$, which is a
contradiction because $\emptyset =X\cap Y=v_{2}(X)\cap
v_{2}(Y)=B\cap v_{2}(\Bbb{P}^{r})\cap v_{2}(Y)=B\cap v_{2}(Y)$. This
proves the claim.

We now claim that $\Phi _{|Y}$ cannot have infinitely many fibres
containing two or more points. In fact, let $Q,Q^{\prime }$ be two
points of $Y$ such that $\Phi (Q)=\Phi (Q^{\prime })$ and let
$L\subset A$ be the line $\langle Q,Q^{\prime }\rangle $. It follows
from Lemma \ref{lemL} that either case $(ii)$ or case $(v)$ holds
for $L$. The common secant lines to $X$ and $Y$ coincide with the
secant lines to $X\cap A$. Since $X$ is defined by quadrics $X\cap
A$ does not contain three points on a line, so the number of common
secant lines to $X$ and $Y$ is equal to $\binom{d}{2}$. We now prove
that for generic $Y\subset A\simeq\Bbb{P}^{c}$ there are at most a
finite number of secant lines $\langle Q,Q^{\prime }\rangle $, with
$Q,Q^{\prime }\in Y$ and $\Phi (Q)=\Phi (Q^{\prime })$, for which
Lemma \ref{lemL} $(v)$ holds. It is here where we strongly use the
assumption $\dim(W)\leq 2n+1$. Consider the rational map $\psi
:\Bbb{P}^{r}\times \Bbb{P}^{r}\dasharrow G(1,r)$ given by $\psi
(Q,Q^{\prime })=\langle Q,Q^{\prime }\rangle $. Let us define
$V:=\{(Q,Q^{\prime })\in \Bbb{P}^{r}\times \Bbb{P}^{r}\mid\psi
(Q,Q^{\prime })\in W,\, \Phi (Q)=\Phi (Q^{\prime })\}$. Then
$\dim(V)\leq 2n+2$ as $\dim (W)\leq 2n+1.$ Therefore, for generic
$Y$, in $\Bbb{P}^{r}\times \Bbb{P}^{r}$ we have $\dim[(Y\times
Y)\cap V]\leq 0$, since $\dim (V)+\dim (Y\times Y)\leq
2n+2+2c-2=2r=\dim (\Bbb{P}^{r}\times \Bbb{P}^{r})$. This proves the
claim.

It follows that $\Phi(Y)\subset\Bbb P^{\alpha}$ has only a finite
number $\eta(Y)$ of singular points. In particular, $B$ intersects the
secant variety of $v_2(Y)$ in $\Bbb P^{N(r)}$ in a
finite number of points. Since the dimension of the secant variety
of $v_2(Y)$ in $\Bbb P^{N(r)}$ is the expected one,
$2\dim[v_2(Y)]+1=2c-1$, we get
$$2c-1=\dim(\sec[v_2(Y)])\leq\cd(B)=\alpha+1$$
whence $\alpha\geq 2c-2$, proving the first statement.

Moreover, the number of singular points of $\pi [v_{2}(Y)]$ is
bounded by $\Delta [v_{2}(Y)]$ so we have
\begin{center}
$\binom{d}{2}\leq\eta(Y)\leq\Delta [v_{2}(Y)]=\binom{2c-1}{c-1}$,
\end{center}
where the equality follows from Corollary \ref{corQ}.

Furthermore, if $\dim(W)\leq 2n$ then $\dim(V)\leq 2n+1$ and hence no
double point of $\Phi(Y)$ comes from a line $L$ for which Lemma
\ref{lemL} (v) holds, arguing as before. Therefore
$\binom{d}{2}=\eta(Y)$. On the other hand, $\eta(Y)=\Delta
[v_{2}(Y)]$ if and only if $\cd(B)=\dim(\sec[v_2(Y)])$, that is, if
and only if $\alpha=2c-2$.
\end{proof}

\begin{remark}
The bound $\binom{d}{2}\leq \binom{2c-1}{c-1}$ is better than the
obvious bound $d\leq 2^c$. In fact, a simple calculation shows that
our bound is given asymptotically by $d\leq
\frac{2^c}{\sqrt[4]{\pi{c}}}$.
\end{remark}

\begin{remark}\label{rem:zak}
If $X\subset\Bbb P^r$ is a non-degenerate integral subvariety (resp.
a finite set of points in general position) defined by quadrics then
it follows from \cite[Corollary 5.4]{z} that $c\leq \alpha+1
\leq\binom{c+1}{2}$. On the one hand, if $c\leq\alpha+1 < 2c-1$ then
$\dim(W)>2n+1$ by Theorem \ref{teofond}, so our method say nothing
about $d$. On the other hand, if $\binom{c}{2} < \alpha+1
\leq\binom{c+1}{2}$ then $d\leq 2c$, and $\alpha+1=\binom{c+1}{2}$
if and only if $d=c+1$ and $X\subset\Bbb P^r$ is a variety of
minimal degree (i.e. either a cone over the Veronese surface
$v_2(\Bbb P^2)\subset\Bbb P^5$ or a rational normal scroll) (cf.
\cite[Proposition 5.6, Corollary 5.8 and Remark 5.9]{z}).
\end{remark}

\begin{remark}
In view of Remark \ref{rem:zak}, our result is more relevant in the
wide range $2c-2\leq \alpha < \binom{c}{2}$. Furthermore, looking at
the proof of Theorem \ref{teofond} one observes that the closer is
$\alpha$ to $2c-2$, the better should be the bound $\binom{d}{2}\leq
\binom{2c-1}{c-1}$. It would be very interesting to find, under
similar hypotheses, a bound on the degree of $X$ involving not only
$c$ but also $\alpha$. For instance, under the assumptions of
Theorem \ref{teofond} it can be shown that
$\binom{d}{2}\leq\binom{2c-1}{c-1} + 2c-2 - \alpha$ if moreover
$\Phi(Y)$ is non-degenerate in $\Bbb P^{\alpha}$. More generally,
$\binom{d}{2}\leq \binom{2c-1}{c-1} + 2c-2 - \beta$, where $\beta$
denotes the dimension of the linear span of $\Phi(Y)$ in $\Bbb
P^{\alpha}$. However, these bounds are probably far from being
optimal.
\end{remark}

\begin{remark}\label{rem:egh}
A bound on the degree of a zero dimensional scheme defined by
quadrics was conjectured in \cite[Conjecture $(II_{m,r})$]{e-g-h}.
In the particular case $2c-2=\alpha$, where our bound turns out to
be stronger, Conjecture $(II_{m,r})$ predicts $d\leq 2^{c-1}+1$. We
would like to remark that we actually get the better bound
$\binom{d}{2}\leq \binom{2c-1}{c-1}$ under the extra assumption
$\dim(W)\leq 1$.
\end{remark}

%
%

In Theorem \ref{teofond}, we assume $X$ is reduced in codimension
zero and $\dim(W)\leq 2n+1$. This is crucial to prove that
$\Phi_{|Y}:Y\to\Phi(Y)$ has only finitely many double points. Let us
see that these hypotheses cannot be dropped in the following two
remarks.

\begin{remark}
In the proof of Theorem \ref{teofond}, we assume $X$ is reduced in
codimension zero to ensure that $X\cap A$ consists of $d$ different
points. In this way, we have finitely many common secant lines to
$X$ and $Y$, whence $\Phi(Y)$ has finitely many double points coming
from lines as in Lemma \ref{lemL} (ii). This is no longer true if
$X\cap A$ is non-reduced, as the following example shows. Consider
$\Lambda\subset H^0(\Bbb P^c,\O_{\Bbb P^c}(2))$ generated by
$X_iX_j$ for $1\leq i\leq j\leq c$. Then $X\subset\Bbb P^c$ is
supported at the point $(1:0:\dots:0)$. In this case, every line
passing through $(1:0:\dots:0)$ is contracted by $\Phi(Y)$. Hence,
for every smooth quadric $Y$ in $\Bbb P^c$ of dimension $c-1$ not
passing through $(1:0:\dots:0)$ the morphism $\Phi_{|Y}:Y\to\Phi(Y)$
is a double covering of $\Phi(Y)=v_2(\Bbb P^{c-1})\subset\Bbb
P^{N(c-1)}$ by $Y$.
\end{remark}

\begin{remark}\label{remW}
On the other hand, the assumption on $W$ is used to guarantee that
$\Phi(Y)$ has finitely many double points coming from lines as in
Lemma \ref{lemL} (v). Unfortunately, this condition cannot be
relaxed. Let us consider the following example. Choose coordinates
$(x:y:z:u:w)$ in $\Bbb{P}^{4}$ and fix the hyperplane $w=0.$ Let
$F_{0},F_{1,}F_{2}$ be three generic degree two forms in
$\Bbb{C}[x,y,z,u]$ such that the intersection of the corresponding
quadrics is given by $8$ distinct points in the hyperplane $w=0 $.
Let $\Phi :\Bbb{P}^{4}\dasharrow \Bbb{P}^{6}$ be the rational map
given by $(F_{0}:F_{1}:F_{2}:wx:wy:wz:wu)$ and let $X$ be the base
locus of the linear systems $\Lambda $ of quadrics giving $\Phi$.
$X$ is the union of the $8$ points and $(0:0:0:0:1)$, so that $r=4$,
$\alpha =6$, $d=9$, $n=0$ and $c=4$. It is easy to see that $\dim
(Z)=4$, where $Z:=\Phi(\Bbb{P}^{4})$, and the fibre over any point
of $Z$ is a point with the exception of points $(h:k:l:0:0:0:0)$
whose fibres are positive dimensional, and the generic fibre is an
elliptic smooth quartic of $\Bbb{P}^{3}$, intersection of two
quadrics. Here $\dim (W)\geq 2$ by Lemma \ref{lemW}, and Theorem
\ref{teofond} does not hold since $36=\binom{9}{2}>\binom{7}{3}=35$.
\end{remark}

\begin{remark}
\label{remZ} In the proof of Theorem \ref{teofond} we showed that
$\Phi_{|Y}:Y\to\Phi(Y)$ is finite. Moreover, since $\dim(W)\leq
2n+1$ and $A:=\langle Y \rangle\subset\p^r$ is generic we deduce
that $\dim(W\cap\g(1,A))\leq 1$. So it easily follows from Lemma
\ref{lemW} that $\Phi_{|A}:A\dashrightarrow\Phi(A)$ is birational
(in particular, if $n=0$ the assumptions of Theorem \ref{teofond}
imply that $\Phi:\Bbb P^r\dashrightarrow\Bbb P^{\alpha}$ is
birational onto its image). Hence $\dim(Z)\geq\dim(\Phi(A))\geq c$,
where $Z:=\Phi(\Bbb P^r)$. As $\dim(Z)=\rho-1$, where $\rho$ is the
generic rank of the Jacobian matrix of $\Lambda$, it follows that a
necessary condition to get $\dim(W)\leq 2n+1$ is $\rho\geq c+1$.
Note that, in concrete examples, to determine $\rho$ is easier than
to estimate the dimension of $W$.


\end{remark}





In practice, it could be difficult to compute the dimension of $W$.
However, as we pointed out in the introduction, we have
$W=\emptyset$ as soon as condition $K_2$ holds.

\begin{corollary}
\label{corK2} Let $X\subset\Bbb{P}^{r}$ be a scheme of degree $d$
and dimension $n$ defined by a linear system of quadrics $\Lambda$.
Assume that $X$ is reduced in codimension zero. Let
$\alpha:=\dim(\Lambda)-1$ and let $c:=r-n\geq 2$. If $X$ satisfies
condition $K_{2}$ then $\alpha \geq 2c-2$ and $\binom{d}{2}\leq
\binom{2c-1}{c-1}$. Furthermore, $\binom{d}{2}=\binom{2c-1}{c-1}$ if
and only if $\alpha=2c-2$.
\end{corollary}

\begin{proof}
If $X$, or $\Lambda$, satisfies condition $K_2$ then the restriction
$\Lambda_{|L}$ to a line $L\subset\Bbb{P}^{r}$ also satisfies $K_2$
(see \cite[Lemma 4.2]{v})). Therefore, if $L\cap X=\emptyset$ then
$\Lambda_{|L}$ is given by the complete linear system of quadrics,
so that $W=\emptyset$. Hence we can apply Theorem \ref{teofond}.
\end{proof}



\begin{remark}
According to Theorem \ref{teofond}, the inequality $\alpha+1\geq
2c-1$ is a necessary condition to have $\dim(W)\leq 2n+1$. As far as
we know, this bound on the number of quadrics defining $X$ was not
known even for schemes satisfying property $N_2$ (cf.
\cite[Corollary 3.7 and Remark 3.9]{h-k}).
\end{remark}

\begin{example}
Generic linear sections of (a cone over) either
$\Bbb{P}^{1}\times\Bbb{P}^{2}\subset\Bbb{P}^{5}$ or
$G(1,4)\subset\Bbb{P}^{9}$ are examples of varieties satisfying
$N_2$ and the equalities obtained in Theorem \ref{teofond} and
Corollary \ref{corK2} with $(d,c)$ equal to either $(3,2)$ or
$(5,3)$, respectively (cf. Remark \ref{rem:dioph}).
\end{example}

\begin{example}\label{ex:8points}
Let $X$ be the scheme given by $8$ points in general position in
$\Bbb{P}^{4}$. $X$ is the base locus of a generic $6$-dimensional
linear system of quadrics $\Lambda$, and property $N_{1}$ holds (see
\cite[Theorem 1]{gl}). Here $c=4$, $\alpha=6$ and $d=8$. Then
$\alpha=2c-2$ but $\binom{8}{2}\neq\binom{7}{3}$, whence $X$ does
not satisfy property $N_2$ by Corollary \ref{corK2}. This shows an
example where \cite[Theorem 1]{gl} is sharp for $p=2$.

The same computation shows that a curve of genus three embedded in
$\p^5$ with degree $8$ does not satisfy property $N_2$. In that
case, it easily follows also from \cite[Theorem 2]{gl}.
\end{example}

\begin{remark}\label{rem:dioph}
The set of pairs of integers $(d,c)$ satisfying the Diophantine
equation $\binom{d}{2}=\binom{2c-1}{c-1}$ is not known (see for
instance \cite{dw}). The pairs $(3,2)$, $(5,3)$ and $(221,9)$ are
solutions, but there could be other ones. However, if the generic
curve section of $X$ is smooth and integral, $(3,2)$ and $(5,3)$ are
the only possibilities for equality in Theorem \ref{teofond} and
Corollary \ref{corK2}, thanks to the following:
\end{remark}

\begin{theorem}\label{teofondter}
In Theorem \ref{teofond}, if $\dim(W)\leq 2n$,
$\binom{d}{2}=\binom{2c-1}{c-1}$ and, moreover, $X$ is smooth and
integral in codimension one then either $d=3$, $c=2$ and $g=0$ or
$d=5$, $c=3$ and $g=1$.
\end{theorem}

\begin{proof}
We consider a generic smooth integral subvariety
$Y\subset\Bbb{P}^{r}$ of dimension $c-1$, disjoint from $X$, such
that $\langle Y\rangle  =:A\simeq\Bbb{P}^{c+1}$. Consider
$\Phi(Y)\subset\Bbb P^{\alpha}$. The double points of $\Phi(Y)$
correspond to lines as in cases (ii) and (v) of Lemma \ref{lemL}.
Since $\dim(W)\leq 2n$, one gets as in Theorem \ref{teofond} that
actually no double point of $\Phi(Y)$ comes from a line as in case
(v). Hence $\Phi(Y)$ has only a finite number $\eta(Y)$ of double
points, and $\eta(Y)$ is the number of common secant lines to $X$
and $Y$, or equivalently, the number of common secant lines to
$X\cap A$ and $Y$. As $Y\subset A$ is generic and $X\cap A$ is a
smooth integral curve by hypothesis, we can compute the number of
common secant lines to $X\cap A$ and $Y$ by using Lemma
\ref{lemsec}. On the other hand, since
$\binom{d}{2}=\binom{2c-1}{c-1}$ it follows from Theorem
\ref{teofond} that $\alpha=2c-2$. Therefore
$\eta(Y)=\Delta[v_2(Y)]$, and the second number can be computed by
Theorem \ref{formula}. As $a_0=\binom{d}{2}$, $a_1=\binom{d-1}{2}-g$
and $b_i=0$ for $i\geq 2$ because the generic $i$-dimensional linear
section of $Y$ has no apparent double points, these two computations
yield
\begin{center}
$\binom{d}{2}b_{0}+[\binom{d-1}{2}-g]b_{1}=\eta(Y)=\binom{2c-1}{c-1}b_0+\binom{2c-1}{c-2}b_1$.
\end{center}

Since $\binom{d}{2}=\binom{2c-1}{c-1}$ we deduce
$\binom{d-1}{2}-g=\binom{2c-1}{c-2}$. By using Castelnuovo's
inequality in $\Bbb{P}^{c+1}$ we have
\begin{center}
$\frac{1}{2}(d^{2}-3d+2)-\binom{2c-1}{c-2}=g \leq \frac{d^{2}-2d+1}{2c}+\frac{d-1}{2}$
\end{center}
(see \cite{acgh} p. 116). As
\begin{center}
$\binom{2c-1}{c-2}=\frac{c-1}{c+1}\binom{2c-1}{c-1}=\frac{c-1}{c+1}\binom{d}{2}$,
\end{center}
we get
\begin{center}
$\frac{1}{2}(d^{2}-3d+2)-\frac{c-1}{c+1}\binom{d}{2}\leq
\frac{d^{2}-2d+1}{2c}+\frac{d-1}{2},$
\end{center}
i.e. $d-2-\frac{c-1}{c+1}d\leq\frac{d-1}{c}+1$, and, equivalently,
$\frac{2d-2c-2}{c+1}\leq\frac{d+c-1}{c}$. So we get
$d\leq\frac{3c^2+2c-1}{c-1}$. From the relation
$\binom{2c-1}{c-1}=\binom{d}{2}$ we have
\begin{center}
$2\binom{2c-1}{c-1} \leq \frac{3c^2+2c-1}{c-1}(\frac{3c^2+2c-1}{c-1}-1)$,
\end{center}
and it is easy to see that this last inequality cannot be satisfied
for $c\geq 6$. If $c\leq 5$, taking care of the relation
$\binom{d}{2}=\binom{2c-1}{c-1}$, then the only possibilities are
$d=3$, $c=2$, $g=0$ and $d=5$, $c=3$, $g=1$.
\end{proof}

\begin{remark}\label{rem:smooth in codimension one}
If $X$ is smooth and integral in codimension one and $\dim (W)\leq
2n+1$, then Lemma \ref{lemsec} and Theorem \ref{formula} yield
\begin{center}
$a_{0}b_{0}+a_{1}b_{1}\leq \eta (Y)\leq\binom{2c-1}{c-1}b_{0}+\binom{2c-1}{c-2}b_{1}$.
\end{center}
Consider $Y:=Y_e\subset\Bbb{P}^{c+1}$ the complete intersection of
two hypersurfaces of degree $e$. Then $b_0(Y_e)=\binom{e^2}{2}$, and
$b_1(Y_e)=\binom{e^2-1}{2}-(e^2(e-2)+1)$ since the sectional genus
of $Y_e$ is $e^2(e-2)+1$. Therefore, for every $e\geq 2$ we have the
relation
$a_{0}\binom{e^2}{2}+a_{1}[\binom{e^2-1}{2}-(e^2(e-2)+1)]\leq
\binom{2c-1}{c-1}\binom{e^2}{2}+\binom{2c-1}{c-2}[\binom{e^2-1}{2}-(e^2(e-2)+1)]$.
Dividing by $\binom{e^2}{2}$ we get $a_{0}+a_{1}f(e)\leq
\binom{2c-1}{c-1}+\binom{2c-1}{c-2}f(e)$, where
$f(e)=\frac{\binom{e^2-1}{2}-(e^2(e-2)+1)}{\binom{e^2}{2}}$. Since
$\lim_{e\to\infty}f(e)=1$ we get $a_{0}+a_{1}\leq
\binom{2c-1}{c-1}+\binom{2c-1}{c-2}$. So we obtain the bound
\begin{center}
$\binom{d}{2}+\binom{d-1}{2}-g\leq\binom{2c-1}{c-1}+\binom{2c-1}{c-2},$
\end{center}
where $g$ denotes the sectional genus of $X\subset\p^r$.
\end{remark}

\begin{remark}
This bound is slightly better than the bound obtained in
Theorem \ref{teofond}. Furthermore, if $X$ is smooth and integral in
higher codimensions the same argument produces (a little bit)
stronger and stronger bounds involving numerical characters of the
corresponding linear section.
\end{remark}

\begin{remark}
A priori, the choice of $Y$ in Theorem \ref{teofond} and Remark
\ref{rem:smooth in codimension one} might seem rather arbitrary.
However, in Theorem \ref{teofond} the same bound is obtained by
taking a hypersurface $Y\subset\p^c$ of any degree. On the other
hand, in Remark \ref{rem:smooth in codimension one} it is natural to
consider a complete intersection in view of Hartshorne's Conjecture,
and the bound does not change if we choose a complete intersection
of hypersurfaces of different degrees. Furthermore, some
computations where $Y\subset\p^{c+1}$ is a non complete intersection
codimension two subvariety, for $c+1\leq 5$, suggest that it is not
possible to improve the bound of Remark \ref{rem:smooth in
codimension one} with our method.
\end{remark}

\begin{corollary}\label{corfin}
In Corollary \ref{corK2}, if moreover $X$ is smooth and integral in
codimension one and $\binom{d}{2}=\binom{2c-1}{c-1}$, then either
$d=3$, $c=2$ and $g=0$ or $d=5$, $c=3$ and $g=1$.
\end{corollary}

\begin{proof}
This is an immediate consequence of Theorem \ref{teofondter}.
\end{proof}

Furthermore, in view of \cite{h-k}, Corollaries \ref{corK2} and
\ref{corfin} yield a stronger result when $X\subset\p^r$ satisfies
property $N_p$ or $N_{2,p}$ for $p\geq 3$:

\begin{proof}[Proof of Corollary \ref{cor:Np}]
If $p\geq 3$, let $r':=r+2-p$ and let $X'\subset\p^{r'}$ denote the
inner projection from $p-2$ general smooth points of $X\subset\p^r$.
Let $d':=d+2-p$ and $c':=c+2-p$ denote, respectively, the degree and
codimension of $X'\subset\p^{r'}$. If $X\subset\p^r$ satisfies
property $N_p$ or $N_{2,p}$ then
$h^0(\Bbb{P}^{r},\mathcal{I}_X(2))\geq cp-\binom{p}{2}$ by
\cite[Corollary 3.7]{h-k}, and $X'\subset\p^{r'}$ satisfies property
$N_{2,2}$ by \cite[Corollary 3.3]{h-k}. Therefore
$\binom{d'}{2}\leq\binom{2c'-1}{c'-1}$ by Corollary \ref{corK2}, and
hence $\binom{d+2-p}{2}\leq\binom{2c+3-2p}{c+1-p}$. Furthermore,
$\binom{d+2-p}{2}=\binom{2c+3-2p}{c+1-p}$ if and only if
$h^0(\Bbb{P}^{r'},\mathcal{I}_{X'}(2))=2c'-1$ by Corollary
\ref{corK2}. Note that this happens if and only if
$h^0(\Bbb{P}^{r},\mathcal{I}_X(2))=cp-\binom{p}{2}$ (cf.
\cite[Proposition 3.5]{h-k}). This proves (i). Furthermore, if
$X\subset\p^r$ is smooth and integral in codimension one we claim
that also $X'\subset\p^{r'}$ is smooth and integral in codimension
one. By induction, it is enough to prove the claim for the inner
projection of $X\subset\p^r$ from a single point. Let $x\in X$ be a
smooth general point and let $C\subset\p^{c+1}$ denote a smooth
integral curve section of $X\subset\p^r$ passing through $x$. Let
$C'\subset\p^{c'+1}$ denote the inner projection of
$C\subset\p^{c+1}$ from $x$. Then $C'\subset\p^{c'+1}$ is a smooth
integral curve section of $X'\subset\p^{r'}$ isomorphic to
$C\subset\p^{c+1}$ since there are no trisecant lines to
$C\subset\p^{c+1}$ passing through $x$, as $X\subset\p^r$ is defined
by quadrics. Therefore if $\binom{d+2-p}{2}=\binom{2c+3-2p}{c+1-p}$,
i.e. if $\binom{d'}{2}=\binom{2c'-1}{c'-1}$, then either $d'=3$,
$c'=2$ and $g'=0$ or $d'=5$, $c'=3$ and $g'=1$ by Corollary
\ref{corfin}, that is, either $d=p+1$, $c=p$ and $g=0$ or $d=p+3$,
$c=p+1$ and $g=1$. This proves (ii).
\end{proof}

\begin{remark}
On the other hand, as it is well known, both rational normal curves
of degree $p+1$ and elliptic normal curves of degree $p+3$ satisfy
property $N_p$ (see, for instance, \cite{g} or \cite{gl}).
\end{remark}

\begin{remark}
The bounds obtained in Corollary \ref{cor:Np} improve those of
Theorem \ref{thm:main} and Corollary \ref{cor:main}. In fact, $d$ is
bounded asymptotically by $\frac{2^{c+2-p}}{\sqrt[4]{\pi(c+2-p)}}$.
\end{remark}

According to Remark \ref{rem:dioph}, it would be interesting to
answer the following:

\begin{question}
Is it possible to obtain $221$ points in $\Bbb{P}^{9}$ as the base
locus of a linear system of quadrics satisfying property $N_2$,
$N_{2,2}$, $K_2$ or $\dim(W)\leq 0$?
\end{question}

\textbf{Acknowledgements.} This research was begun while the second
author was visiting the Department of Mathematics at Universit\`a
degli Studi di Milano during the spring of 2008. He wishes to thank
Antonio Lanteri for his warm hospitality and for providing excellent
working conditions. We also thank F.L. Zak for pointing out the
paper \cite{e-g-h} quoted in Remark \ref{rem:egh}, as well as for
letting us know about \cite{h-k}.


%
%
%

\end{document}